\documentclass[submission,copyright,creativecommons]{eptcs}
\providecommand{\event}{ACT2019-190} 
\usepackage{breakurl}             
\usepackage{underscore}           

\usepackage[utf8]{inputenc}
\usepackage{graphicx}
\usepackage{amssymb}
\usepackage{amsmath}
\usepackage{amsthm}
\usepackage{mathrsfs}
\usepackage{bbm}
\usepackage{bm}
\usepackage{color}
\usepackage{enumitem}

\usepackage[english]{babel}
\usepackage[T1]{fontenc}
\usepackage{mathtools}

\newtheorem{Definition}{Definition}
\newtheorem{Lemma-Definition}{Lemma-Definition}
\newtheorem{Proposition}{Proposition}

\newtheorem{Lemma}{Lemma}
\newtheorem{Theorem}{Theorem}
\newtheorem{Remark}{Remark}

\long\def\symbolfootnote[#1]#2{\begingroup
\def\thefootnote{\fnsymbol{footnote}}\footnote[#1]{#2}\endgroup}
\def\mybig{\big}

\def\vec{\boldsymbol}
\def \obs{y}

\def\xvec{\vec x}

\newcommand \rv[1]{\mathcal{M} \left( \{0,1 \} \times \Omega,\Rset^{#1} \right)}

\def\Pbb{{\mathbb P}}

\def\Rset{\mathbb{R}}
\def\Nset{\mathbb{N}}

\def\Rset{\mathbb{R}}
\def\Nset{\mathbb{N}}
\def\tol{\tau}
\newcommand{\mynorm}[1][\cdot]{ |  #1 |}

\def\group{\syms}
\def\transform{{s}}

\def\intervalzeroone{(  0 ,  1 )}
\def\level{\gamma}
\def\subclassdet{\mathfrak{C}^{[\SampleSize]}_\level}
\def\Signal{\vec{\Theta}}
\def\family{\mathfrak C}

\def\Hcal{\mathcal{H}}

\def\Scal{\mathcal{S}}

\def\Xfrak{\mathfrak{X}}

\def\Mcal{\mathcal M}
\def\Ncal{\mathcal N}

\def\test{f}
\def\testa{f}
\def\testb{g}

\newcommand{\Bcal}{\mathcal B}

\newcommand{\PFA}[1]{\Pbb_{\textsc{fa}} \left [ {#1} \right ] }
\newcommand{\PDET}[1]{\Pbb_{\textsc{det}} \left [  {#1} \right ] }

\newcommand{\ElementaryObs}{Y}

\newcommand{\Obs}[2]{\ElementaryObs_{#2}}
\newcommand{\SeqObs}[1]{{\ElementaryObs}}
\newcommand{\VecObs}[2]{\vec{\ElementaryObs}_{\! \! #2} }

\newcommand{\nupleObs}[2]{ \big ( \IdealObs{#1}{1},\IdealObs{#1}{2}, \ldots, \IdealObs{#1}{#2} \big ) }

\newcommand{\Fcal}{\mathcal F}

\newcommand{\preorda}{\, \, {\ord} \, \, }

\newcommand{\preordc}{\, \, {\ord}^{\!*} \, \,}
\newcommand{\preordd}{\, \, {\ord}^{\diamond} \, \,}
\newcommand{\ord}{\preceq}

\newcommand{\Tests}[1]{\mathcal M \left( \Rset^{#1}, \{ 0,1 \} \right)}
\newcommand{\Decisions}{\mathcal M \! \left( \{0,1 \} \! \times \! \Omega, \big \{0,1 \big \} \right)}

\newcommand{\decision}{{D}}
\newcommand{\Topt}[1]{\test^{\textsc{np}(\level)}_{#1}}
\newcommand{\TRDT}[2]{\test^{\textsc {rdt}(\level,#2)}_{#1}}

\newcommand{\normcdf}{\Phi}
\newcommand{\Class}{\Ccal_\level}

\newcommand{\Acal}{\mathcal A}

\newcommand{\god}{\decision^*}

\newcommand{\id}{\textrm{id}}
\newcommand{\Noise}{\Xvec}

\newcommand{\ElementaryNoise}{X}

\newcommand{\tribu}{\Bcal}
\newcommand{\Ecal}{\mathcal E}

\newcommand{\reduc}[1]{\Xi}

\newcommand{\SetOfOracles}{\SymbolSetOfDecisions^*_{\level}}

\newcommand{\mydef}{\, \, := \, \, }
\newcommand{\SetOfSeqs}[1]{\textup{Seq}_{\, #1}}
\newcommand{\constant}{q}
\newcommand{\Interf}{\Delta}

\newcommand{\Qcal}{\mathcal Q}

\newcommand{\Mean}[2]{\langle {#1} \rangle_{#2}}

\newcommand{\Threshold}[1]{\lambda_\level(#1)}
\newcommand{\maj}{\textup{upper}}

\newcommand{\Sinfty}{\Bcal_\infty}

\newcommand{\IdealObs}[2]{Y_{#2}}
\newcommand{\Landscape}[1]{\SymbolLandscape_\level\left( {#1} \right)}
\newcommand{\landscape}{\SymbolLandscape}

\newcommand{\SymbolLandscape}{\Dfrak}
\newcommand{\SymbolSetOfLandscapes}{\Dbb}

\newcommand{\SetOfLandscapes}{\SymbolSetOfLandscapes_\level}
\newcommand{\SubsetOfLandscapes}[1]{\SymbolSetOfLandscapes^{#1}_{\level}}

\newcommand{\SetOfNPLandscapes}{\SymbolSetOfLandscapes^{\textup{NP}(\level)}}
\newcommand{\SetOfRDTLandscapes}[1]{\SymbolSetOfLandscapes^{\textup{RDT}(\level)}_{#1}}

\newcommand{\WholeSetOfDecisions}{\Decisions}
\newcommand{\Power}[1]{2^{#1}}

\newcommand{\PowerClass}{\Power{{\Class}}}

\newcommand{\intervalc}{J}
\newcommand{\constantinterval}{[0,1/2)}
\newcommand{\PowerSetOfOracles}{\Power{\SetOfOracles}}
\newcommand{\Selectivity}[1]{\textup{Sel}_\level \left( #1 \right)}

\newcommand{\ParameterSet}{\Xfrak_\level}
\newcommand{\parameter}{\xi}
\newcommand{\syms}{\Scal}
\newcommand{\action}{\Acal}
\newcommand{\SampleSize}{n}
\newcommand{\ConstantSet}{\Qcal}

\def\tol{\tau}

\def\I{\mathbf{I}}
\def\Bcal{\mathcal{B}}

\def\Ncal{\mathcal{N}}

\def\Hcal{\mathcal H}

\def\Pbb{\mathbb{P}}
\def\Rset{\mathbb{R}}
\def\Nset{\mathbb{N}}

\def\mybig{\big}

\def\level{\gamma}

\def\mybig{\big}

\newcommand{\classdet}{\mathcal{K}^{[{\SampleSize}]}_\level}

\newcommand{\size}[1]{\alpha^{[\SampleSize]}_{#1}}

\newcommand{\power}[1]{\beta^{[\SampleSize]}_{#1}}

\def\Signal{\Theta}

\def\Sig{\Signal}

\def\intervalzeroone{(  0 ,  1 )}

\def\family{\mathfrak F}

\newcommand{\rtest}{\overline{\test}}

\def\Identity{\I_\SampleSize}

\def\Noise{W}
\def\process{Z}

\def\rp{\Mcal(\Omega,\Rset^\SampleSize)}

\newcommand{\cat}[1]{\mathscr{#1}} 
\newcommand{\catc}{\cat{C}}
\newcommand{\ncat}[1]{\mathbf{#1}} 
\newcommand{\setdef}[1]{\left\lbrace #1 \right\rbrace}
\newcommand{\setdeftwo}[2]{\setdef{ #1 \ \vert \ #2 }}

\newcommand{\Cocones}[1]{\ncat{Cocones}\left( #1 \right)}

\renewcommand{\SymbolLandscape}{\mathfrak{L}}

\renewcommand{\Class}{\mathrm{Dec}_{\level}}
\renewcommand{\SetOfOracles}{ \mathbb{O}_{\level} }
\renewcommand{\Landscape}[1]{\mathrm{ Lnd }_\level\left( {#1} \right)}
\renewcommand{\SetOfLandscapes}{{\mathrm{ Lndscp }_{\level}}}
\newcommand{\SetOfLandscapesPrime}{{\mathrm{ Lndscp}'_{\level}}}
\renewcommand{\SymbolSetOfLandscapes}{\mathrm{ Lndscps }}
\renewcommand{\SetOfNPLandscapes}{\mathrm{ Lnd }^{\textup{NP}(\level)}}
\renewcommand{\SetOfRDTLandscapes}[1]{\mathrm{ Lnd }^{\textup{RDT}(\level,#1)}}
\renewcommand{\classdet}{\mathrm{Tests}^{[{\SampleSize}]}_\level}

\title{Interfacing biology, category theory and mathematical statistics}
\author{Dominique Pastor${\:}^{1}$ \quad\quad Erwan Beurier${\:}^{1}$
\institute{1 \ \ IMT Atlantique \\ LabSTICC \& LEGO \\ Universit\'{e} de Bretagne-Loire \\ 29238 Brest, France}
\email{\{ erwan.beurier ; dominique.pastor ; roger.waldeck \}@imt-atlantique.fr}
\and
Andrée Ehresmann${\:}^{2}$ \qquad\qquad Roger Waldeck${\:}^{1}$
\institute{2 \ \ Faculté des Sciences \\Dépt. de Mathématiques \\ LAMFA \\ Université de Picardie Jules Verne\\ 33 rue Saint-Leu \\ F-80039 Amiens, France}
\email{andree.ehresmann@u-picardie.fr}
}
\def\titlerunning{Interfacing biology, CT and statistics}
\def\authorrunning{Pastor, Beurier, Ehresmann \& Waldeck}
\begin{document}
\maketitle

\begin{abstract}
Motivated by the concept of degeneracy in biology \cite{Edelman1973}, we establish a first connection between the Multiplicity Principle \cite{EhresmannVanbremeersch2007, Ehresmann_2019} and mathematical statistics. Specifically, we exhibit two families of tests that satisfy this principle to achieve the detection of a signal in noise.
\end{abstract}

\section{Introduction}

In \cite{Edelman1973}, Edelman \& Gally pointed out degeneracy as the fundamental property allowing for living systems to evolve through natural selection towards more complexity in fluctuating environments. Degeneracy is defined \cite{Edelman1973} as “ \emph{… the ability of elements that are structurally different to perform the same function or yield the same output}”. 
Degeneracy is a crucial feature of immune systems and neural networks, at all organization levels. 

The Multiplicity Principle (MP) \cite{EhresmannVanbremeersch2007, Ehresmann_2019}, introduced by Ehresmann \& Vanbremeersch, is a mathematical formalization of degeneracy in Categorical terms. The consequences of this principle, as treated in \cite{EhresmannVanbremeersch2007, Ehresmann_2019}, underpin Edelman \& Gally’s conjecture according to which “\emph{complexity and degeneracy go hand in hand}” \cite{Edelman1973}. 

Another property of many biological and social systems is their resilience: (i) they can perform in degraded mode, with some performance loss, but without collapsing; (ii) they can recover their initial performance level when nominal conditions are satisfied again; (iii) they can perform corrections and auto-adaption so as to maintain essential tasks for their survival. In addition, resilience of social or biological systems is achieved via agents with different skills. For instance, cells are simply reactive organisms, whereas social agents have some cognitive properties. Thence the idea that resilience may derive from fundamental properties satisfied by agents, interactions and organizations. Could this fundamental property be a possible consequence of degeneracy \cite[Section 3.1, p. 15]{Ehresmann_2019}?

The notion of resilience remains, however, somewhat elusive, mathematically speaking. In contrast, the notion of robustness has a long history and track record in mathematical statistics \cite{Hampel86}. By and large, a statistical method is robust if its performance is not unduly altered in case of outliers or fluctuations around the model for which it is designed. Can we fathom the links between resilience and robustness?

As an attempt to embrace the questions raised above from a comprehensive outlook, the original question addressed in this work is the possible connection between MP and robustness to account for emergence of resilience in complex systems. 
As a first step in our study aimed at casting the notions of robustness, resilience and degeneracy within the same theoretical framework based on the MP, we hereafter establish that statistical tests do satisfy the MP. The task to perform by the tests is the fundamental problem of detecting a signal in noise. However, to ease the reading of a paper at the interface between category theory and mathematical statistics, we consider a simplified version of this problem. 

The paper is organized as follows. We begin by specifying notation and notions in mathematical statistics. In Section \ref{Section: general MP}, we state the MP in categorical words on the basis of \cite{EhresmannVanbremeersch2007} and consider the particular case of preorders, which will be sufficient at the present time to establish that statistical tests satisfy the MP for detecting signals in noise. In Section \ref{Section: problem statement}, we set out the statistical detection problem. We will then introduce, in Section \ref{Subsection: preorder}, a preorder that makes it possible to exhibit two types of "structurally different" tests, namely, the Neyman-Pearson tests (Section \ref{Section: NP}) and the RDT tests (Section \ref{Section: RDT}). Section \ref{Section: MP} concludes the paper by establishing that these two types of tests achieve the MP for the detection problem under consideration. For space considerations, we limit proofs to the minimum making it possible to follow the approach without too much undue effort.

\subsection*{Summary of main results}

Because this paper lies at the interface between different mathematical specialties, the present section summarizes its contents in straight text. To begin with, the MP is a property that a category may satisfy when it involves structurally different diagrams sharing the same cocones. To state our main results, it will not be necessary to consider the general MP though. In fact, the particular case of preordered sets will suffice, in which case the MP reduces to Proposition~\ref{proposition: multiplicity principle in a preorder}.
\\
\indent
Second, in statistical hypothesis testing, a hypothesis can be seen as a predicate, of which we can aim at determining the truth value by using statistical decisions. 
There exist many optimality criteria to devise a decision to test a given hypothesis. In non-Bayesian approaches, which will be our focus below, such criteria are specified through the notions of size and power. 

The size is the least upper bound for the probability of rejecting the hypothesis when this one is actually true. We generally want this size to remain below a certain value called level, because the hypothesis to test mostly represents the standard situation. For instance, planes in the sky are rare events, after all, and the standard hypothesis is "there is no plane", which represents the nominal situation. A too large level may result in an intolerable cluttering of a radar screen.

We do not want to be bothered by too many alarms. In contrast, when the hypothesis is false, we want to reject it with the highest possible confidence. The probability that a decision rejects the hypothesis when this one is actually false is called the power of the decision. For a given testing problem, we thus look for decisions with maximal power within the set of those decisions that have a size less than or equal to a a specified level. This defines a preorder. A maximal element in this preorder is said to be optimal. 

Different hypotheses to test may thus require different criteria, specified through different notions of size and different notions of power. This is what we exploit below to exhibit two sets of "structurally different" decisions that satisfy the MP.

To carry out this construction, we consider the detection of a signal in independent standard gaussian noise, a classical problem in many applications. This is an hypothesis testing problem for which there exists an optimality criterion where the size is the so-called probability of false alarm and the power is the so-called probability of detection. This criterion has a solution, the Neyman-Pearson (NP) decision, which is thus the maximal element of a certain preorder. We can consider a second class of decisions, namely, the RDT decisions. These decisions are aimed at detecting deviations of a signal with respect to a known deterministic model in presence of independent standard gaussian noise. This problem is rotationally invariant and the RDT decisions are optimal with respect to a specific criterion defined through suitable notions of size and power. They are maximal elements of another preordered set. Although not dedicated to signal detection, these decisions can be used as surrogates to NP decisions to detect a signal. It turns out that the family of RDT decisions and that of NP decisions satisfy the MP as stated in Theorem \ref{Theorem: MP}. This is because the more data we have, the closer to perfection both decisions are. 

\subsection*{Notation}

\noindent
\textbf{Random variables.} Given two measurable spaces $\Ecal$ and $\Fcal$, $\Mcal(\Ecal,\Fcal)$ denotes the set of all measurable functions defined on $\Ecal$ and valued in $\Fcal$. The two $\sigma$-algebra involved are omitted in the notation because, in the sequel, they will always be obvious from the context. In particular, we will throughout consider a probability space $(\Omega, \tribu,\Pbb)$ and systematically endow $\Rset$ with the Borel $\sigma$-algebra, which will not be recalled. Therefore, $\Mcal(\Omega,\Rset)$ designates the set of all real random variables and $\Mcal(\Omega,\Rset^n)$ is the set of $n$-dimensional real random vectors. 

Given $\constant \in [0,\infty[$, $\Sinfty(\constant)$ is the set of all real random variables $\Interf \in \Mcal(\Omega,\Rset)$ such that ${\left| \Interf \right|_{\infty}} \leqslant \constant$. As usual, we write $X \thicksim \Ncal(0,1)$ to mean that $X \in \Mcal(\Omega,\Rset)$ is standard normal. 
Given a sequence $(X_n)_{n \in \Nset} \in \Mcal(\Omega,\Rset)^\Nset$ of real random variables, we write $X_1, X_2, \ldots \stackrel{\text{iid}}{\thicksim} \Ncal(0,1)$ to mean that $X_1, X_2, \ldots$ are independent and identically distributed with common distribution $\Ncal(0,1)$. 
\medskip \\
\noindent
\textbf{Decisions et Observations.} Throughout, $\Decisions$ designates the set of all measurable functions $\decision: \big \{0,1\big \} \times \Omega \to \big \{ 0,1 \big \}$. Any element of $\Decisions$ is called a decision for obvious reasons given below. If $\decision \in \WholeSetOfDecisions$ then, for any $\varepsilon \in \big \{0,1\big \}$, $\decision(\varepsilon)$ denotes the Bernoulli-distributed random variable $\decision(\varepsilon): \Omega \to \big \{0,1\big \}$ defined for any given $\omega \in \Omega$ by $\decision(\varepsilon)(\omega) = \decision(\varepsilon,\omega)$. An $n$-dimensional test is hereafter any measurable function $\test: \Rset^n \to \{0,1 \}$ and $\Tests{n}$ stands for the set of all $n$-dimensional tests. A measurable function $X: \left \{ 0,1 \right \} \times \Omega \to \Rset^n$ is hereafter called an observation and $\rv{n}$ denotes the set of all these observations. Given a test $\test \in \Tests{n}$ and $X \in \rv{n}$, $\decision = \test(X)$ is trivially a decision: $\decision \in \WholeSetOfDecisions$. If $X \in \rv{n}$ then, for any $\varepsilon \in \{ 0,1 \}$, $X(\varepsilon) = X(\varepsilon,\cdot) \in \Mcal(\Omega,\Rset^n)$ is defined for every $\omega \in \Omega$ by $X(\varepsilon)(\omega) = X(\varepsilon, \omega)$.
\medskip \\
\noindent
\textbf{Empirical means.} We define the empirical mean of a given sequence $y = (y_n)_{n \in \Nset}$ of real values as the sequence $\left( \Mean{y}{n} \right)_{n \in \Nset}$ of real values such that, $\forall n \in \Nset, \Mean{y}{n} \mydef \frac{1}{n} \sum_{i=1}^n y_i$. By extension, the empirical mean of a sequence $Y = (Y_n)_{n \in \Nset}$ of random variables where each $Y_n \in \Mcal(\Omega, \Rset)$ is the sequence $\left( \Mean{Y}{n} \right)_{n \in \Nset}$ of random variables where, for each $n \in \Nset$, $\Mean{Y}{n} \in \Mcal(\Omega,\Rset)$ is defined by $\Mean{Y}{n} \mydef \frac{1}{n} \sum_{i=1}^n Y_i$. Therefore, for any $\omega \in \Omega$, $\Mean{Y}{n}(\omega) \mydef \Mean{Y(\omega)}{n}$ with $Y(\omega) = (Y_n(\omega))_{n \in \Nset}$. If $Y = (Y_n)_{n \in \Nset}$ is a sequence of observations ($\forall \, n \in \Nset$, $Y_n \in \rv{}$), we define the empirical mean of $Y$ as the sequence $\left( \Mean{Y}{n} \right)_{n \in \Nset}$ of observations such that, for $\varepsilon \in \{0,1\}$, $\Mean{Y}{n} \in \rv{}$ with $\Mean{Y}{n}(\varepsilon) = \Mean{Y(\varepsilon)}{n}$ and $Y(\varepsilon) = (Y_n(\varepsilon))_{n \in \Nset}$.
\medskip \\
\noindent
\textbf{Preordered sets.}
Given a preordered set $(E,\ord)$ and $A \subset E$, the set of maximal elements of $A$ is denoted by $\max \left( A, (E,\ord) \right)$, the set of upper bounds of $A$ is denoted by $\maj \left( A, (E,\ord) \right)$ and the set of least upper bounds of $A$ in $(E,\ord)$ is denoted by $\sup \left( A, (E,\ord) \right)$.

\section{Multiplicity Principle}
\label{Section: general MP}


\subsection{General case}


The multiplicity principle (MP) comes from \cite{EhresmannVanbremeersch2007}. It proposes a categorical approach to the biological degeneracy principle, which ensures a kind of flexible redundancy. Roughly, MP, in a category $\catc$, ensures the existence of \emph{structurally non isomorphic} diagrams with same colimit. 
A formal definition relies on the notion of a cluster between diagrams in a category $\catc$.

%
%
%

\begin{Definition}[Cluster]
	\label{definition: cluster}
	
	Let $D: \cat{D} \to \catc$ and $E: \cat{E} \to \catc$ be two (small) diagrams. A cluster $G: D \to E$ is a maximal set $G = \setdeftwo{f: D(d) \to E(e)}{ d \in \cat{D}, e \in \cat{E} , f \in \catc}$ such that: 
	
		\vspace{-2mm}
	
	\begin{enumerate}[itemsep=-1mm]
		\item[(i)]		for all $d \in \cat{D}$ there exist $e \in \cat{E}$ and $g: D(d) \to E(e)$ such that $g \in G$
		
		\item[(ii)]		let $G(d)$ be the subset of $G$ consisting of arrows $g: D(d) \to E(e)$ associated to the same $d$; then $G(d)$ is included in a connected component of the comma-category $\left(D(d) \ | \ E\right)$
		
		\item[(iii)]	if $g: D(d) \to E(e) \in G(d)$ and $\varepsilon: e \to e' \in \cat{E}$, then $E(\varepsilon) \circ g \in G(d)$
		
		\item[(iv)]		if $\delta: d' \to d \in \cat{D}$ and $g: D(d) \to E(e) \in G(d)$, then $g \circ D(\delta) \in G(d')$
	\end{enumerate}
	
	%
\end{Definition}

For instance, a connected cone from $c$ to $D$ can be seen as a cluster from the constant functor $\Delta(c)$ to $D$; and any cocone from $E$ to $c$ is a cluster $E \to \Delta(c)$.

\begin{Remark}
	\label{Remark: free cocompletion of C}
	Adjacent clusters can be composed: a cluster $G: D \to E_0$ and a cluster $G_0: E_0 \to E$ can be composed to a cluster $G_0 \circ G$. We can then consider a category of clusters of $\catc$, whose objects are the (small) diagrams $\cat{D} \to \catc$, and an arrow $D \to E$ is a cluster. This category is isomorphic to the free cocompletion of $\catc$ \cite{EhresmannVanbremeersch2007}. 
	
	%
	%
	%
	%
	
\end{Remark}

A cluster $G: D \to E$ defines a functor $\Omega G: \Cocones{E} \to \Cocones{D}$ mapping a cocone $\alpha$ to the cocone $\alpha \circ G$ (composite of $\alpha$, seen as a cluster, and $G$, which is a cluster).

\begin{Definition}[Multiplicity principle (MP)]
	\label{definition: multiplicity principle}
	A category \ $\catc$ \emph{satisfies the multiplicity principle (MP)} if there exist two diagrams $D: \cat{D} \to \catc$ and $E: \cat{E} \to \catc$ such that:
	
	\vspace{-1mm}
	
	\begin{enumerate}[itemsep=-1pt]
		\item[(i)] $\Cocones{D} \cong \Cocones{E}$;
		
		\item[(ii)] There is no cluster $G: D \to E$ nor $G: E \to D$ such that $\Omega G$ is an isomorphism. 
	\end{enumerate}
\end{Definition}

$D$ and $E$ having the same cocones translates the property of both systems to accomplish the same function. The absence of clusters  between $D$ and $E$ that define an isomorphism, reflects the structural difference between $D$ and $E$, which is key to robustness and adaptability: if the system described by $E$ fails, then $D$ may replace it.

\subsection{Application to preorders}

The main purpose of this paper is to find a meaningful instance of the MP in some preorder. In the following, we do not distinguish between a preorder and its associated category.

\begin{Proposition}[MP in a preorder]
	\label{proposition: multiplicity principle in a preorder}
	Let $\left(E, \leqslant \right)$ be a preorder. If there are two disjoint subsets $A,B \subset E$ such that the following conditions hold, then $E$ verifies the MP: 
	
	\vspace{-1mm}
	
	\begin{enumerate}[itemsep=-1pt]
		\item[(i)] $A$ and $B$ have the same sets of upper bounds
		
		\item[(ii)] There is an $a \in A$ with no upper bounds in $B$
		
		\item[(iii)] There is a $b \in B$ with no upper bounds in $A$
	\end{enumerate}
\end{Proposition}

\begin{proof}
	Condition (i) ensures that $A$ and $B$ have isomorphic categories of cocones. Conditions (ii) and (iii) respectively ensure that there is no cluster $i_{A} \to i_{B}$ nor $i_{B} \to i_{A}$ where $i_{A}: A \hookrightarrow E$ and $i_{B}: B \hookrightarrow E$ are the inclusion functors.
\end{proof}

Albeit trivial, the following lemma will be helpful.
\begin{Lemma}
	\label{Lemma: PM elementaire}
	Given a preordered set $(E,\preceq)$, if $A$ and $B$ are two subsets of $E$ such that $\, A \times \! B \, \, \cap \ord \, \, = \, \emptyset$ and $\sup \left( A, (E,\ord) \right) = \sup \left( B , (E,\ord) \right)$, then $E$ satisfies the MP.
\end{Lemma}

\section{Statistical detection of a signal in noise}
\label{Section: problem statement}

\subsection{Problem statement}
Let $\varepsilon \in \{0,1\}$ be the unknown indicator value on whether a certain physical phenomenon has occurred ($\varepsilon = 1$) or not ($\varepsilon = 0$). We aim at determining this value. It is desirable to resort to something more evolved than tossing a coin to estimate $\varepsilon$. However, whatever $\decision$, the decision is erroneous for any $\omega \in \Omega$ such that $\decision(\varepsilon,\omega) \neq \varepsilon$. We thus have two distinct cases.
\medskip \\
\noindent\textbf{False alarm probability:}
If $\varepsilon= 0$ and $\decision(0,\omega) = 1$, we commit a false alarm or error of the $1^{\text{st}}$ kind, since we have erroneously decided that the phenomenon has occurred while nothing actually happened. We thus define the false alarm probability (aka size, aka error probability of the $1^{\text{st}}$ kind) of $\decision$ as:

\vspace{-2mm}

\begin{equation}
\label{Eq: PFA}
\PFA{ \decision } \stackrel{\text{def}}{=} \Pbb \mybig [ \decision(0) = 1 \mybig ]
\end{equation}

\noindent\textbf{Detection probability:}
If $\varepsilon = 1$ and $\decision(1,\omega) = 0$, we commit an error of the $2^{\text{nd}}$ kind, also called missed detection since, in this case, we have missed the occurrence of the phenomenon. As often in the literature on the topic, we prefer to use the probability of correctly detecting the phenomenon and we define the detection probability as:

\vspace{-2mm}

\begin{equation}
\label{Eq: PDET}
\PDET{\decision} \stackrel{\text{def}}{=} \Pbb \mybig [ \decision(1) = 1 \mybig ]
\end{equation}

\subsection{Decision with level $\level \in \intervalzeroone$ and oracles}

Among all possible decisions, the omniscient oracle $\god \in \WholeSetOfDecisions$ is defined for any pair $(\varepsilon, \omega) \in \big \{0,1\big \} \times \Omega$ by setting $\god(\varepsilon, \omega) = \varepsilon$. Its probability of false alarm is $0$ and its probability of detection is $1$: $\PFA{\god} = 0$ et $\PDET{\god} = 1$. This omniscient oracle has no practical interest since it knows $\varepsilon$. That's not really fair! Since it is not possible in practice to guarantee a null false alarm probability, we focus on decisions whose false alarm probabilities are upper-bounded by a real number $\level \in \intervalzeroone$ called level. We state the following definition.
\begin{Definition}[Level]
	\label{Def: niveau}
	Given $\level \in ]0,1[$, we say that $\decision \in \WholeSetOfDecisions$ has level $\level$ if $\, \PFA{ \decision} \leqslant \level$. The set of all decisions with level $\level \in ]0,1[$ is denoted by $\Class$.
\end{Definition}

We can easily prove the existence of an infinite number of elements in $\Class$ that all have a detection probability equal to $1$. Whence the following definition.

\begin{Definition}
	Given $\level \in \intervalzeroone$, an oracle with level $\level$ is any decision $\decision \in \Class$ such that $\PDET{\decision} = 1$. The set of all the oracles with level $\level$ is denoted by $\SetOfOracles$.
\end{Definition}

Oracles with level $\level$ have no practical interest either since they require prior knowledge of $\varepsilon$! Therefore, we restrict our attention to decisions in $\Class$ that "approximate" at best the oracles with level $\level$, without prior knowledge of $\varepsilon$, of course. To this end, we must preorder decisions. 

\begin{Lemma-Definition}[Total preorder $(\Class, \preorda)$]
	\label{Def: pre-ordre total}
	For any given $\level \in \intervalzeroone$ and any pair $(\decision, \decision') \in \Class \times \Class$, we define a preorder $(\Class, \preorda)$ by setting:
	
	\vspace{-2mm}
	
	\begin{equation}
	\decision \preorda \decision' \quad \text{if} \quad \PDET{\decision} \leqslant \PDET{\decision'}.
	\end{equation}	
	We write $\decision \cong \decision'$ if $\decision \preorda \decision'$ and $\decision' \preorda \decision$.
\end{Lemma-Definition}

\subsection{Observations}

In practice, observations help us decide whether the phenomenon has occurred or not. By collecting a certain number of them, we can expect to make a decision. Hereafter, observations are assumed to be elements of $\rv{}$ and corrupted versions of $\varepsilon$. We suppose that we have a sequence $(\ElementaryObs_n)_{n \in \Nset}$ of such random variables. As a first standard model, we could assume that, for any $n \in \Nset$ and any $(\varepsilon, \omega) \in \{0,1\} \times \Omega$, $\ElementaryObs_n(\varepsilon, \omega) = \varepsilon + \ElementaryNoise_n(\omega)$ with $\ElementaryNoise_1, \ElementaryNoise_2, \ldots, \ElementaryNoise_n, \ldots \stackrel{\text{iid}}{\thicksim} \Ncal(0,1)$. In this additive model, $\ElementaryNoise_n$ models noise on the $n$th observation. We could make this model more complicated and realistic by considering random vectors instead of variables. However, with respect to our purpose, the significant improvement we can bring to the model is elsewhere. Indeed, we have assumed above that the signal, regardless of noise, is $\varepsilon$. However, from a practical point of view, it is more realistic to assume that the $n$th observation $\ElementaryObs_n$ captures $\varepsilon$ in presence of some interference $\Interf_n$, independent of $\ElementaryNoise_n$. In practice, the probability distribution of $\Interf_n$ will hardly be known and, as a means to compensate for this lack of knowledge, we assume the existence of a uniform bound on the amplitude of all possible interferences. Therefore, we assume that, for all $(\varepsilon,\omega) \in \{0,1\} \times \Omega$, $\ElementaryObs_n(\varepsilon,\omega) = \varepsilon + \ElementaryNoise_n(\omega) + \Interf_n(\omega)$ and the existence of $\constant \in [0,\infty)$ such that $\Interf_n \in \Sinfty(\constant)$. 
After all, this model is standard in time series analysis: $\varepsilon$ plays the role of a trend, $\Interf_n$ is the seasonal variation and $\ElementaryNoise_n$ is the measurement noise. 

For each $\constant \in [0,\infty)$, $\SetOfSeqs{\constant}$ henceforth designates the set of all the sequences:
\begin{equation}
	\nonumber
	\SeqObs{\constant} = \left( \ElementaryObs_n \right)_{n \in \Nset} \in \WholeSetOfDecisions^{\! \Nset }
\end{equation}
such that, $\forall n \in \Nset$ and $\forall (\varepsilon,\omega) \in \{0,1\} \times\Omega$, $\ElementaryObs_n(\varepsilon,\omega) = \varepsilon + \Interf_n(\omega) + \ElementaryNoise_n(\omega)$, where $\Interf_n \in \Sinfty(\constant)$ and $\ElementaryNoise_n \thicksim \Ncal(0,1)$ are independant. Therefore, for all $n \in \Nset$ and all $\varepsilon \in \{0,1\}$, $\ElementaryObs_n(\varepsilon) = \varepsilon + \Interf_n + \ElementaryNoise_n$, with $\ElementaryNoise_1, \ElementaryNoise_2, \ldots, \ElementaryNoise_n, \ldots \stackrel{\text{iid}}{\thicksim} \Ncal(0,1)$.

\section{Selectivity, landscapes of tests and preordering}
\label{Subsection: preorder}

For any sequence $\SeqObs{\constant} = \left( \ElementaryObs_n \right)_{n \in \Nset} \in \WholeSetOfDecisions$, we henceforth set:

\vspace{-2mm}

\begin{equation}
\label{Eq: vec of seq}
\VecObs{\constant}{n} = \nupleObs{\constant}{n}
\end{equation}

\noindent
In other words, $\VecObs{\constant}{n}$ is the truncated version of the original sequence
$\SeqObs{\constant}$ at the $n$th term.

\begin{Definition}[Selectivity of a test]
	\label{Definition: Selectivite d'un test}
	Given any $n \in \Nset$ and any test $\test \in \Tests{n}$, the selectivity of $\test$ at given level $\level \in \intervalzeroone$ is defined as the set:
	
	\vspace{-2mm}
	
	\begin{equation}
	\label{Eq: Selectivité}
	\Selectivity{\test} \mydef \Big \{ \, \constant \in \constantinterval: \forall \, \SeqObs{\constant} \in \SetOfSeqs{\constant}, \test \left( \VecObs{\constant}{n} \right)\in \Class \, \Big \}
	\nonumber
	\end{equation}
\end{Definition}

\noindent 
The relevance of the interval $\constantinterval$ in the definition above will pop up in Section \ref{subsec: Application to Detection}.

\begin{Definition}[Landscapes of tests]
	\label{Definition: Paysage d'un test}
	Given any $n \in \Nset$ and any test $\test \in \Tests{n}$, the landscape of $\test$ at given level $\level \in \intervalzeroone$ is the subset of $\, \Class$ defined by:
	
	\vspace{-2mm}
	
	\begin{equation}
	\label{Eq: Landscape}
	\Landscape{\test} \mydef \bigcup_{\constant \in \Selectivity{\test}} \Big \{ \test \left( \VecObs{\constant}{n} \right) : \SeqObs{\constant} \in \SetOfSeqs{\constant} \Big \} 
	\end{equation}
	
	\noindent The total landscape covered by all the tests $\test \in \Tests{n}$ is defined by setting:
	\begin{equation}
	\label{Eq: Complete set of landscapes}
	\SetOfLandscapes \, \, \mydef \, \, \bigcup_{n \in \Nset} \, \Big \{ \, \Landscape{\test}: \test \in \Tests{n} \, \Big \}
	\end{equation}
	%
\end{Definition}


This notion of landscape makes it possible to compare tests via the following preorder. 
The proofs that the following definition is consistent and that the next lemma holds true are left to the reader.

\begin{Definition}[Preorder $(\PowerClass \, , \preordc)$]
	\label{Def: pre-ordre dans SetOfLandscapes}
	Given any level $\level \in \intervalzeroone$, we define the preorder $(\PowerClass \, , \preordc)$ via the following three properties:
	\medskip \\
	\noindent 
	\textbf{(P1)} $\forall n \in \Nset$, $\forall (\testa,\testb) \in \Tests{n} \times \Tests{n}$, $\Landscape{\testa} \preordc \Landscape{\testb}$ if:
	
	\vspace{-2mm}
	
	\begin{equation}
	\label{Eq: preordc(i)}
	\nonumber
	\Selectivity{\testa} = \Selectivity{\testb} \quad \text{and} \quad \forall \, \constant \in \Selectivity{\testa}, \forall \, \SeqObs{\constant} \in \SetOfSeqs{\constant}, \testa \left( \VecObs{\constant}{n} \right) \, \preorda \, \, \testb \left( \VecObs{\constant}{n} \right)
	\end{equation}
	\noindent
	\textbf{(P2)} $\forall \, (\landscape, \landscape') \in \left( \SetOfLandscapes \cup \PowerSetOfOracles \right) \! \! \times \PowerSetOfOracles$, \ $\landscape \preordc \landscape'$ 
	\medskip \\
%
%
	\noindent
	\textbf{(P3)} $\forall \, \landscape \in \PowerClass \setminus \left( \SetOfLandscapes \cup \PowerSetOfOracles \right)$, $\landscape \, \preordc \, \landscape$
\end{Definition}

\begin{Lemma}
	\label{Lemma: preordc inclus dans preorda}
	$\forall \, (\landscape,\landscape') \in \left( \SetOfLandscapes \cup \PowerSetOfOracles \right) \times \left( \SetOfLandscapes \cup \PowerSetOfOracles \right)$, $\landscape \preordc \landscape' \Rightarrow \landscape \times \landscape' \subset \preorda.$
\end{Lemma}

With this material, we can state our first result that will prove useful in applications to statistical decisions below.

\begin{Theorem}[Approximation of oracles in $(\PowerClass,\! \preordc)$]
	\label{Theorem: Approximation des oracles}
	Given $\level \in \intervalzeroone$, if a set $\ParameterSet$ and a family of tests $\left( \test_{\parameter,n} \right)_{\parameter \in \ParameterSet, n \in \Nset}$ satisfy: 
	\medskip \\
	(i) $\forall (\parameter,n) \in \ParameterSet \times \Nset$, $\test_{\parameter,n} \in \Tests{n}$; \medskip \\
	(ii) $\exists \ConstantSet_\level \subset [0,\infty)$, $\forall (\parameter,n) \in \ParameterSet \times \Nset$, $\Selectivity{f_{\parameter,n}} = \ConstantSet_\level$; \medskip \\
	(iii) $\forall (\parameter,\constant) \in \ParameterSet \times \ConstantSet_\level$, $\forall \SeqObs{\constant} \in \SetOfSeqs{\constant}$, $\displaystyle \lim_{n \to \infty} \PDET{\test_{\parameter,n} \left( \VecObs{\constant}{n}\right)} = 1$; \medskip \\
	then, by setting $\SetOfLandscapesPrime = \Big \{ \, \Landscape{\test_{\parameter,n}}: n \in \Nset, \parameter \in \ParameterSet \, \Big \}$, we have: 
	
	\vspace{-2mm}
	
	\begin{equation}
	\label{Eq: Oracles = borne sup des tests avec les bonnes propriétés}
	\PowerSetOfOracles = \maj \left( \SetOfLandscapesPrime \, \, , \, \big ( \PowerClass, \preordc \big ) \right)
	= \sup \left( \SetOfLandscapesPrime \, \, , \, \big ( \PowerClass, \preordc \big ) \right)
	\end{equation}
	
\end{Theorem}
\begin{proof}
	For any $(\parameter,n) \in \ParameterSet \times \Nset$ and any $\landscape \in \PowerSetOfOracles$, (P2) in Definition \ref{Def: pre-ordre dans SetOfLandscapes} straightforwardly implies that $\Landscape{\test_{\parameter,n}} \preordc \landscape $. As a consequence:
	
	\vspace{-2mm}
	
	\begin{equation}
	\label{Eq: Subset of landcsapes majoré par oracles}
	\PowerSetOfOracles \subset \maj \left( \SetOfLandscapes' \, \, , \, \big ( \PowerClass, \preordc \big ) \right) 
	\end{equation}
	To prove the converse inclusion, consider some $\landscape \in \maj \left( \SetOfLandscapes' \, \, , \, \big ( \PowerClass, \preordc \big ) \right)$. We thus have $\forall \, (\parameter,n) \in \Nset \times \ParameterSet$, $\Landscape{\test_{\parameter,n}} \preordc \landscape$. According to Lemma \ref{Lemma: preordc inclus dans preorda}, we have $\forall \, ({\parameter,n}) \in \ParameterSet \times \Nset$, $\Landscape{\test_{\parameter,n}} \times \landscape \subset \preorda$. Therefore, $\forall \, ({\parameter,n}) \in \Nset \times \ParameterSet$, $\forall \, \constant \in \Selectivity{\test_{\parameter,n}}$, $\forall \, \SeqObs{\constant} \in \SetOfSeqs{\constant}$ and $\forall \, \decision \in \landscape$, $\test_{\parameter,n} \left( \VecObs{\constant}{n} \right) \preorda \landscape$. It follows from the definition of $\preorda$ and assumption (ii) above that: 
	$$\forall \, (\parameter,n) \in \Nset \times \ParameterSet, \forall \, \constant \in \ConstantSet_\level, \forall \, \SeqObs{\constant} \in \SetOfSeqs{\constant}, \forall \, \decision \in \landscape, \PDET{ \test_{\parameter,n} \left( \VecObs{\constant}{n} \right) } \leqslant \PDET{\decision}$$ We derive from assumption (iii) that $\PDET{\decision} = 1$ and thus that $\decision \in \SetOfOracles$. It follows that $\landscape \in \PowerSetOfOracles$. We obtain that $\maj \left( \SetOfLandscapes' \, \, , \, \big ( \PowerClass, \preordc \big ) \right) \subset \PowerSetOfOracles$ and therefore, from \eqref{Eq: Subset of landcsapes majoré par oracles}, $\PowerSetOfOracles = \maj \left( \SetOfLandscapes' \, \, , \, \big ( \PowerClass, \preordc \big ) \right)$. The second equality in \eqref{Eq: Oracles = borne sup des tests avec les bonnes propriétés} is straightforward since the elements of $\PowerSetOfOracles$ are isomorphic in the sense of $\preordc$.
\end{proof}

For later use, given $\intervalc \subset [0,\infty)$, $n \in \Nset$ and $\family \subset \Tests{\SampleSize}$, we hereafter set: 

\vspace{-2mm}

\begin{equation}
\label{Eq: Subset of landscapes}
\SubsetOfLandscapes{\, \intervalc}(\family) \mydef \Big \{ \Landscape{\test} \in \SetOfLandscapes: \test \in \family \, , \, \Selectivity{\test} = \intervalc \Big \}
\end{equation}

\section{The Neyman-Pearson (NP) solution}
\label{Section: NP}

When $n$ spans $\Nset$, the Neyman-Pearson (NP) Lemma makes it possible to pinpoint a maximal element in each $(\SubsetOfLandscapes{\{0\}}(\family),\preorda)$ with $\family = \Tests{\SampleSize}$. These maximal elements are hereafter called NP decisions. Specifically, we have the following result.

\begin{Lemma}[Maximality of the NP decisions]
	\label{Lemma: NP}
	For any $\level \in \intervalzeroone$ and any $n \in \Nset$, 
	
	\vspace{-2mm}
	
	\begin{equation}
	\label{Eq: optimalité des tests NP} 
	\Landscape{\Topt{n}} = \max \left( \SubsetOfLandscapes{\{0\}
	}\left( \Tests{\SampleSize}\right), \preordc \right)
	\end{equation}
	where $\Topt{n} \in \Tests{n}$ is the $n$-dimensional NP test with size $\level$ defined by:
	
	\vspace{-2mm}
	
	\begin{equation}
	\label{Eq:TNP}
	\forall (\obs_1, \obs_2, \ldots, \obs_n) \in \Rset^n, \Topt{n}(\obs_1, \obs_2, \ldots, \obs_n)  = 
	\left \{
	\begin{array}{lll}
	1 & \text{if \, $\sum_{i=1}^{n} \obs_{i} > \sqrt{n} \, \Phi^{-1}(1-\level)$} \medskip \\
	0 & \text{otherwise}
	\end{array}
	\right.
	\end{equation}
	and satisfies, $\forall \, \SeqObs{0} \in \SetOfSeqs{0}$, 
	$$
	\left \{
	\begin{array}{lll}
	\PFA{ \Topt{n} \left( \VecObs{0}{n} \right) } = \level \medskip \\
	\PDET{ \Topt{n} \left( \VecObs{0}{n} \right) } = 1 - \Phi \mybig ( \, \Phi^{-1}(1-\level) - \sqrt{n} \, \mybig )
	\end{array}
	\right.
	$$
\end{Lemma}
\begin{proof}
	A direct application of the Neyman-Pearson Lemma \cite[Theorem 3.2.1, page 60]{Lehmann2005}, followed by some standard algebra to obtain $\PDET{ \Topt{n} \left(\VecObs{0}{n}\right) }$. 	
\end{proof}

The next result states that it suffices to increase the number of observations to approximate oracles with level $\level$ by NP decisions.

\begin{Theorem}[Approximation of oracles with level $\level$ by NP decisions in $(\PowerClass,\! \preordc)$]
	\label{Theorem: Approximation des oracles par decisions NP}
	$ $ \medskip \\
	Setting $\, \SetOfNPLandscapes \mydef \Big \{ \, \Landscape{\Topt{n}} : n \in \Nset \, \Big \}$ for any $\level \in \intervalzeroone$, we have:
	
	\vspace{-2mm}
	
	\begin{equation}
	\label{Eq: Oracles = borne sup des tests NP}
	\nonumber
	\PowerSetOfOracles = \maj \left( \SetOfNPLandscapes \, \, , \, \big ( \PowerClass, \preordc \big ) \right)
	= \sup \left( \SetOfNPLandscapes \, \, , \, \big ( \PowerClass, \preordc \big ) \right)
	\end{equation}
\end{Theorem}
\begin{proof}
	Given $\level \in \intervalzeroone$, set $\ParameterSet = \{0\}$ and, $\forall \, n \in \Nset, f_{0,n} = \Topt{n}$. According to Lemma \ref{Lemma: NP}: 
	\begin{equation}
		\nonumber
		\displaystyle \lim_{n \to \infty} \PDET{ \Topt{n} (\VecObs{0}{n}) } = 1
	\end{equation}
	Thence the result as a consequence of Theorem \ref{Theorem: Approximation des oracles}.
\end{proof}

\section{The RDT solution}
\label{Section: RDT}

\subsection{An elementary RDT problem}

\noindent
\textbf{Problem statement.} The RDT theoretical framework is exposed in full details in \cite{RDT, RDTlm}. To ease the reading of the present paper, we directly focus on the particular RDT problem that can be used in connection with the detection problem at stake. 

In this respect, suppose that $\process = \Signal + \Noise \in \rp$, where $\Signal$ and $\Noise$ are independent elements of $\rp$. In the sequel, we assume that $\Noise \thicksim \Ncal(0,\Identity)$, $\Identity$ being the $\SampleSize \times \SampleSize$ identity matrix, and consider the mean testing problem of deciding on whether $\mynorm[\Mean{\Signal}{\SampleSize}(\omega)] \leqslant \tol$ (null hypothesis $\Hcal_0$) or $\mynorm[\Mean{\Signal}{\SampleSize}(\omega)] > \tol$ (alternative hypothesis $\Hcal_1$), when we are given $\process(\omega) = \Signal(\omega) + \Noise(\omega)$, for $\omega\in \Omega$. The idea is that $\Signal$ oscillates uncontrollably around $0$ and that only sufficient large deviations of the norm should be detected. This is a particular Block-RDT problem, following the terminology and definition given in \cite{RDTlm}. This problem is summarized by dropping $\omega$, as usual, and writing:

\vspace{-2mm}

\begin{equation}
\label{Eq:Block-RDT event testing problem}
\left \{ \! \! \! \!
\begin{array}{lll}
\, \text{\textbf{Observation:}} \, \process = \Signal + \Noise \in \rp 
\text{\, with } 
\left \{ \! \! \! 
\begin{array}{lll}
\Signal \in \rp,  
\Noise \thicksim \Ncal(0,\I_N), 
\\
\text{$\Signal$ and $\Noise$ are independent}, \\
\end{array}
\right. \\
\begin{array}{lll}
\! \Hcal_0: \, \, \mynorm[ \Mean{\Signal}{\SampleSize} ] \leqslant \tol , \\
\! \Hcal_1: \, \, \mynorm[ \Mean{\Signal}{\SampleSize} ] > \tol .
\end{array}
\end{array}
\right.
\end{equation}

Standard likelihood theory \cite{Lehmann2005, Borovkov, Eaton1983} does not make it possible to solve this problem. Fortunately, this problem can be solved as follows via the Random Distortion Testing (RDT) framework.
\medskip \\
\textbf{Size and power of tests for mean testing.} We seek tests with guaranteed size and optimal power, in the sense specified below. 

\begin{Definition}[Size for the mean testing problem]
	\label{Definition: RDT Size}
	The size of $\test \in \Tests{\SampleSize}$ for testing the empirical mean of the signals $\Sig \in \rp$ such that $\Pbb \mybig [ \mynorm[\Mean{\Sig}{\SampleSize}] \leqslant \tol \mybig ] \ne 0$, given $\process = \Sig + \Noise \in \rp$ with $\Noise$ independent of $\Sig$, is defined by:
	
	\vspace{-2mm}
	
	\begin{equation}
	\label{Eq:size for the Block-RDT problem}
	\size{}(\test) = \sup_{\Sig \in \rp \, : \, \Pbb \, [ \, \mynorm[\Mean{\Sig}{\SampleSize} ] \leqslant \tol \, ] \ne 0} \Pbb \mybig [ \test(\process) = 1 \, \mybig \vert \,  \mynorm[\Mean{\Sig}{\SampleSize}] \leqslant \tol \mybig]
	\end{equation}
	We say that $\test \in \Tests{\SampleSize}$ has level (resp. size) $\level$ if $\size{}(\test) \leqslant \level$ (resp. $\size{}(\test) = \level$). The class of all the tests with level $\level$ is denoted by $\classdet$: $$\classdet = \left \{ \test \in \Tests{\SampleSize}: \size{}(\test) \leqslant \level \right \}$$
\end{Definition}

\begin{Definition}[Power for the mean testing problem]
	\label{Definition: RDT Power}
	The power of $\test \in \Tests{\SampleSize}$ for testing the empirical mean of $\, \, \Signal \in \rp$ such that $\Pbb \mybig [ \mynorm[\Mean{\Signal}{\SampleSize}] > \tol \mybig ] \ne 0$ when we are given $\process = \Signal + \Noise \in \rp$, with $\Noise$ independent of $\, \, \Signal$, is defined by:
	
	\vspace{-2mm}
	
	\begin{equation}
	\power{\Signal}(\test) = \Pbb \mybig [ \test(\process) = 1 \, \mybig \vert \,  \mynorm[\Mean{\Signal}{\SampleSize}] > \tol \mybig]
	\end{equation}
\end{Definition}

\noindent
\textbf{The RDT solution.} With the same notation as above, we can easily construct a preorder $\left( \classdet, \preordd \right)$ by setting:
	\vspace{-2mm}
\begin{align}
\forall \, (\test,\test') \in \classdet \times \classdet, \nonumber \\
\test \preordd \test' \, \, \text{if} \, \, \, \forall \, \, \, \Signal & \in \rp, \Pbb \mybig [ \mynorm[\Mean{\Signal}{\SampleSize}] > \tol \mybig ] \ne 0 \Rightarrow \power{\Signal}(\test) \leqslant \power{\Signal}(\test')
\nonumber
\end{align} 
No maximal element exists in $\left( \classdet, \preordd \right)$. However, we can exhibit $\subclassdet \subset \classdet$ whose elements satisfy suitable invariance properties with respect to the mean testing problem and prove the existence of a maximal element in $\left( \subclassdet,\preordd \right)$.

Set $\syms = \big \{ \id,-\id \big \}$ where $\id$ is the identity of $\Rset$. Endowed with the usual composition law $\circ$ of functions, $(\syms,\circ)$ is a group. Let $\action$ be the group action that associates to each given $\transform \in \group$ the map $\action_\transform: \Rset^\SampleSize \rightarrow \Rset^\SampleSize$ defined for every $x = (x_1,x_2, \ldots, x_\SampleSize) \in \Rset^\SampleSize$ by $\action_\transform(x) = ( \transform(x_1), \transform(x_2), \ldots, \transform(x_\SampleSize)).$
Readily, the mean testing problem is invariant under the action of $\action$ in that $\action_\transform(\process) = \action_\transform(\Signal) + \Noise'$ where $\Noise' = (\Noise'_1, \Noise'_2, \ldots, \Noise'_\SampleSize) \thicksim \Ncal(0,\I_\SampleSize)$ is independent of $\action_\transform(\Signal)$. Therefore, $\action_\transform(\process)$ satisfies the same hypotheses as $\process$. We also have $\mynorm[\Mean{\action_\transform(\Signal)}{\SampleSize}] = \mynorm[\Mean{\Signal}{\SampleSize}]$. Hence, the mean testing problem remains unchanged by substituting $\action_\transform(\Signal)$ for $\Signal$ and $\Noise'$ for $\Noise$. It is thus natural to seek $\action$-invariant tests, that is, tests $\test \in \Tests{\SampleSize}$ such that $\test(\action_\transform(x)) = \test(x)$ for any $\transform \in \group$ and any $x \in \Rset^\SampleSize$. 

On the other hand, since we can reduce the noise variance by averaging observations, we consider $\action$-invariant integrator tests, that is, $\action$-invariant tests $\test \in \Tests{\SampleSize}$ for which exists $\rtest \in \Tests{1}$, henceforth called the reduced form of $\test$, such that $\test(\xvec) = \rtest(\Mean{\xvec}{\SampleSize})$ for any $\xvec \in \Rset^\SampleSize$. Reduced forms of $\action$-invariant integrator tests are also $\action$-invariant: $\forall x \in \Rset$, $\forall \transform \in \action$, $\rtest(\transform(x)) = \rtest(x)$. We thus define $\subclassdet \subset \classdet$ as the class of all $\action$-invariant integrator tests with level $\level$. We thus have $\test \in \subclassdet$ if: \medskip \\
\noindent
\textbf{[Size]}: $\size{}(\test) \leqslant \level$; \medskip \\
\textbf{[$\action$-invariance]}: $\forall \, (\transform,x) \in \group \times \Rset^\SampleSize$, $\test \left( \action_\transform(x) \right) = \test(x)$; \medskip \\
\textbf{[Integration]}: $\exists \, \rtest \in \Tests{1}$, $\forall \, x \in \Rset^\SampleSize$, $\test(x) = \rtest(\Mean{x}{\SampleSize})$.
\medskip \\ 
The following result derives from the foregoing and \cite{RDT, RDTlm}. 
\begin{Proposition}[Maximal element of $\left( \subclassdet, \preordd \right)$]
	\label{prop:Block RDT}
	For any $\level \in \intervalzeroone$ and any $\SampleSize \in \Nset$, 
	
	\vspace{-2mm}
	
	\begin{equation}
	\label{Eq: optimalité des tests RDT} 
	\left \{ \TRDT{\SampleSize}{\tol} \right \} = \max \left( \subclassdet, \preordd \right)
	\end{equation}
	where $\TRDT{\SampleSize}{\tol} \in \Tests{\SampleSize}$ is defined by setting
	
	\vspace{-2mm}
	
	\begin{equation}
	\label{Eq:TRDT(2)}
	\forall (y_1, y_2, \ldots, y_\SampleSize) \in \Rset^n, \TRDT{\SampleSize}{\tol}(y_1, y_2, \ldots, y_\SampleSize) = 
	\left \{
	\begin{array}{lll}
	1 & \text{if \, \, $\left \vert \sum_{i=1}^n \obs_i \right \vert \leqslant \sqrt{\SampleSize} \, \Threshold{\tol \sqrt{\SampleSize}}$} \medskip \\
	0 & \text{otherwise}
	\end{array}
	\right.
	\nonumber
	\end{equation}
	and $\Threshold{\tol \sqrt{\SampleSize}}$ is the unique solution in $x$ to the equation $$2 - \normcdf(x-\tol \sqrt{\SampleSize}) - \normcdf(x + \tol \sqrt{\SampleSize}) = \level$$
	where $\normcdf$ is the cumulative distribution function (cdf) of the $\Ncal(0,1)$ law.
\end{Proposition}

\vspace{0.2cm}
RDT and NP tests are structurally different because dedicated to two different testing problems and optimal with respect to two different criteria. This structural difference will be enhanced by coming back to our initial detection problem.

\subsection{Application to Detection}
\label{subsec: Application to Detection}
Consider again the problem of estimating $\varepsilon \in \{0,1\}$, when we have a sequence $\SeqObs{\constant} \in \SetOfSeqs{\constant}$ of observations such that:

\vspace{-2mm}

\begin{equation}
\label{Eq: Suite de va de la famille(2)}
\forall \, n \in \Nset, \forall (\varepsilon,\omega) \in \big \{ 0,1 \big \} \times \Omega, \Obs{\constant}{\SampleSize} (\varepsilon, \omega) = \varepsilon + \Interf_\SampleSize(\omega) + \ElementaryNoise_\SampleSize(\omega)
\end{equation}
where $X_1, X_2, \ldots \stackrel{\text{iid}}{\thicksim} \Ncal(0,1)$ and $\forall \SampleSize \in \Nset$, $\Interf_\SampleSize \in \Sinfty(\constant)$ with $\constant \in [0,\infty)$. The empirical mean of $\Obs{\constant}{}$ satisfies: $\forall \, \SampleSize \in \Nset, \Mean{\Obs{\constant}{} \, }{\SampleSize}(\varepsilon) = \Mean{\Obs{\constant}{}(\varepsilon)}{\SampleSize} = \varepsilon + \Mean{{\Interf}}{\SampleSize} + \Mean{\ElementaryNoise}{\SampleSize}$.
We thus have $\vert \Mean{{\Interf}}{n} \vert \leqslant \constant$ (a-s). Set $\Signal_\SampleSize = \varepsilon + \Interf_\SampleSize$ for every $\SampleSize \in \Nset$. In the sequel, we assume $\constant < 1/2$ because, in this case, we straightforwardly verify that

\vspace{-2mm}

\begin{equation}
\label{Eq: equivalence RDT and det}
\left \{
\begin{array}{lll}
\varepsilon = 0 & \Leftrightarrow & \mynorm[\Mean{\Signal}{\SampleSize}] \leqslant \constant \\
\varepsilon = 1 & \Leftrightarrow & \mynorm[\Mean{\Signal}{\SampleSize}] \geqslant 1 - \constant \\
\end{array}
\right.
\end{equation}
Therefore, when $\constant \in \constantinterval$, deciding on whether $\varepsilon$ is zero or not when we are given $\VecObs{\constant}{n}(\omega)$ amounts to testing whether $\mynorm[\Mean{\Signal}{n}(\omega)] \leqslant \tol$ or not for $\tol \in [\constant,1-\constant]$. We thus can use the decision $\TRDT{\SampleSize}{\tol} \left( \VecObs{\constant}{\SampleSize} \right)$, where $\TRDT{\SampleSize}{\tol}$ is given by Proposition \ref{prop:Block RDT}.


We can calculate the false alarm probability \eqref{Eq: PFA} of $\TRDT{n}{\tol} \left( \VecObs{\constant}{n} \right)$
where $\VecObs{\constant}{n}$ is defined by \eqref{Eq: vec of seq}. The theoretical results in \cite{RDT} yield that $\forall \tol \in [\constant , 1 - \constant], \PFA{\TRDT{n}{\tol} \left( \VecObs{\constant}{n} \right) } \leqslant \level.$
In the sequel, for the sake of simplifying notation, we assume that both $\tol$ and $\constant$ are in $\constantinterval$. In this case, we have:

\vspace{-2mm}

\begin{equation}
\label{Eq: Selectivité de TRDT(2)}
\forall \tol \in \constantinterval \, , \, \, 
\left \{
\begin{array}{lll}
\Selectivity{\TRDT{n}{\tol}} = [0,\tol] \medskip \\
\Landscape{\TRDT{n}{\tol}} = \displaystyle 
\bigcup_{\constant \in [0,\tol]} \Big \{ \TRDT{n}{\tol} \left( \VecObs{\constant}{n} \right): \SeqObs{\constant} \in \SetOfSeqs{\constant} \Big \} 
\end{array}
\right.
\end{equation}
We can then state the following lemma, which is the counterpart to Lemma \ref{Lemma: NP}.

\begin{Theorem}[Maximality of RDT decisions]
	\label{Theorem: optimality of Block-RDT tests for decision}
	For any $\level \in \intervalzeroone$, any $\SampleSize \in \Nset$ and any $0 \leqslant \constant \leqslant \tol < 1/2$, $\Landscape{\TRDT{\SampleSize}{\tol}} = \max \left( \SubsetOfLandscapes{[0,\tol]} \left( \subclassdet \right), \preordc \right)$.
\end{Theorem}

\begin{proof}
	It results from Definition \ref{Definition: Paysage d'un test} that $\Landscape{\test} \mydef \Big \{ \test \left( \VecObs{\constant}{n} \right) : \SeqObs{\constant} \in \SetOfSeqs{\constant}, \constant \in [0,\tol] \Big \}$.
	According to \eqref{Eq: Subset of landscapes}, we also have: $$\SubsetOfLandscapes{[0,\tol]}\left( \subclassdet \right) = \Big \{ \Landscape{\test} \in \SetOfLandscapes: \test \in \subclassdet \, , \, \Selectivity{\test} = [0,\tol] \Big \}$$
	Given $\constant \in [0,\tol]$ and $\SeqObs{\constant} \in \SetOfSeqs{\constant}$, set:
	$$
	\left \{
	\begin{array}{lllll}
	\process & = & \VecObs{\constant}{\SampleSize} = \left( \SeqObs{\constant}_1, \SeqObs{\constant}_2, \ldots, \SeqObs{\constant}_\SampleSize \right) \ \text{(see \eqref{Eq: vec of seq})} \\
	\Noise & = & (X_1, X_2, \ldots, X_\SampleSize) \thicksim \Ncal(0,\Identity) \\
	\Signal & = & (1 + \Delta_1, 1+ \Delta_2, \ldots, 1 + \Delta_\SampleSize)
	\end{array}
	\right.
	$$
	We basically have $\process = \Signal + \Noise$. Consider now the mean testing problem \eqref{Eq:Block-RDT event testing problem} with $\Signal$, $\Noise$ and $\process$ defined as above. For any $\test \in \Tests{\SampleSize}$, it follows from Eqs. \eqref{Eq: Suite de va de la famille(2)} , \eqref{Eq: equivalence RDT and det}, \eqref{Eq: PDET} and \eqref{Eq:size for the Block-RDT problem} that:
	
	\vspace{-2mm}
	
	\begin{equation}
	\label{Eq: PDET and RDT power}
	\power{\Signal}(\test) = \PDET{\test \left( \VecObs{\constant}{\SampleSize} \right) }
	\end{equation}
	Suppose now that $\test \in \subclassdet$ with $\Selectivity{\test} = [0,\tol]$. We derive from Proposition \ref{prop:Block RDT}, \eqref{Eq: PDET and RDT power} and its application to $\TRDT{\SampleSize}{\tol}$, that $\PDET{\test \left( \VecObs{\constant}{\SampleSize} \right)} \leqslant \PDET{\TRDT{\SampleSize}{\tol} \left( \VecObs{\constant}{\SampleSize} \right)}.
	$
	Since $\constant \leqslant \tol < 1/2$ implies that $\constant \in \Selectivity{\test}$ and since $\Selectivity{\test} = \Selectivity{\TRDT{\SampleSize}{\tol}} = [0,\tol]$, we can rewrite the foregoing equality as $\test \left( \VecObs{\constant}{\SampleSize} \right) \preorda \TRDT{\SampleSize}{\tol} \left( \VecObs{\constant}{\SampleSize} \right)$.
	This holding true for any $\constant \in \Selectivity{\test}$, any $\SeqObs{\constant} \in \SetOfSeqs{\constant}$ and since $\test$ and $\TRDT{\constant}{\SampleSize}$ have same selectivity $[0,\tol]$, we derive from the foregoing and Definition \ref{Def: pre-ordre dans SetOfLandscapes} that $\Landscape{\test} \preordc \Landscape{\TRDT{\SampleSize}{\tol}}.$
\end{proof}
We now prove that the oracles with level $\level$ are approximated by RDT decisions.
\vspace{2mm}
\begin{Lemma}[Approximation of oracles with $\level$ by RDT decisions in $(\PowerClass,\preordc)$]
	\label{Lemma: Approximation des oracles par RDT}
	$ $ \medskip \\
	Setting $\SetOfRDTLandscapes{\tol} \mydef \Big \{ \, \Landscape{\TRDT{n}{\tol}} : n \in \Nset \, \Big \}$ for any given $\level \in \intervalzeroone$, we have:
	
	\vspace{-2mm}
	
	\begin{equation}
	\label{Eq: Oracles = borne sup des tests RDT}
	\nonumber
	\PowerSetOfOracles = \maj \left( \SetOfRDTLandscapes{\tol} \, \, , \, \big ( \PowerClass, \preordc \big ) \right)
	= \sup \left( \SetOfRDTLandscapes{\tol} \, \, , \, \big ( \PowerClass, \preordc \big ) \right)
	\end{equation}
\end{Lemma}
\begin{proof}
	Given $\level \in (0,1)$, it follows from \eqref{Eq: PDET} and \cite[Theorem 2]{RDT} that:
	
	\vspace{-2mm}
	
	\begin{equation}
	\label{Eq: PDETRDT(1)}
	\begin{array}{ccc}
	\forall (\constant,\tol) \in \constantinterval \times \constantinterval, \forall n \in \Nset, 
	\PDET{\TRDT{n}{\tol} \left( \VecObs{\constant}{n} \right) 
	} \geqslant Q_{1/2} \big ( ( 1 - \constant) \sqrt{n}, \Threshold{\tol \sqrt{n}} \big )
	\end{array}
	\nonumber
	\end{equation}
	Since $\tol < 1 - \constant$, \cite[Eq. (3) and Lemma B.2]{khanduri:IEEE} induce that $\displaystyle \lim_{n \to \infty} \PDET{\TRDT{n}{\tol} \left( \VecObs{\constant}{n} \right) } = 1$.
	The set $\SetOfRDTLandscapes{\tol} \subset \SetOfLandscapes$ thus satisfies Theorem \ref{Theorem: Approximation des oracles} conditions with $\ParameterSet = \{\tol\}$ and $\forall \, n \in \Nset, f_{n,\tol} = \TRDT{n}{\tol}$.
\end{proof}

\section[Multiplicity Principle]{Multiplicity Principle in $(\PowerClass,\preordc)$}
\label{Section: MP}

To state the MP in $(\PowerClass,\preordc)$, we need the following lemma.

\begin{Lemma}[Selectivity of NP tests]
	\label{Lemme: Sélectivité des tests NP}
	$\forall \SampleSize \in \Nset$, $\Selectivity{\Topt{\SampleSize}} = \{0\}$
\end{Lemma}
\begin{proof}
	A consequence of \cite[Section B, p. 6.]{RDT}.
\end{proof}

We have now all the material to state the main result.

\begin{Theorem}[Multiplicity Principle in $(\PowerClass,\preordc)$] 
	\label{Theorem: MP}
	For any given $\tol \in (0,1/2)$, the MP is satisfied in $(\PowerClass,\preordc)$ by the pair $\left(\SetOfNPLandscapes,\SetOfRDTLandscapes{\tol}\right)$.
\end{Theorem}
\begin{proof}
	According to Theorems \ref{Theorem: Approximation des oracles par decisions NP} and \ref{Theorem: optimality of Block-RDT tests for decision}, the subsets $\SetOfNPLandscapes$ and $\SetOfRDTLandscapes{\tol}$ of $\PowerClass$ are such that $$\sup \left( \SetOfNPLandscapes \, \, , \, \big ( \PowerClass \, , \preordc \big ) \right) = \sup \left( \SetOfRDTLandscapes{\tol} \, \, , \, \big ( \PowerClass \, , \preordc \big ) \right) = \PowerSetOfOracles$$
	
	In addition, \eqref{Eq: Selectivité de TRDT(2)} and Lemma \ref{Lemme: Sélectivité des tests NP} imply that $\SetOfNPLandscapes \times \SetOfRDTLandscapes{\tol} \cap \preordc = \emptyset$. The conclusion follows from Lemma \ref{Lemma: PM elementaire}.
\end{proof}

\section{Conclusions and Perspectives}
\label{Section: Conclusions}

In this paper, via the framework provided by the Multiple Principle (MP), which is motivated by the concept of degeneracy in biology, and by introducing the notions of test landscapes and selectivity, we have established that this principle is satisfied when we consider the standard NP tests and the RDT tests applied to a detection problem. One interest of this result is that it opens prospects on the construction of Memory Evolutive Systems \cite{EhresmannVanbremeersch2007,Ehresmann_2019} via tests.

More elaborated statistical decision problems should be considered beyond this preliminary work. Sequential tests are particularly appealing because they collect information till they can decide with guaranteed performance bounds. On the one hand, the Sequential Probability Ratio Test (SPRT) established in \cite{Wald1945} is proved to be optimal; on the other hand, in \cite{khanduri:IEEE}, we have exhibited non-optimal tests with performance guarantees in presence of interferences. In the same way as NP and RDT tests satisfy PM, we conjecture that these two types of sequential tests satisfy MP as well. 

From a pratical point of view, such results open new prospects for the design of networks of sensors, where combining different types of sensors and tests satisfying the MP could bring resilience to the overall system.

\section*{Acknowledgements}

The authors are very grateful to the reviewers for their strong encouragements and insightful remarks that help improve the readiness of this paper. 

\bibliographystyle{eptcs}
\bibliography{bibibi, generic}

\begin{thebibliography}{10}
\providecommand{\bibitemdeclare}[2]{}
\providecommand{\surnamestart}{}
\providecommand{\surnameend}{}
\providecommand{\urlprefix}{Available at }
\providecommand{\url}[1]{\texttt{#1}}
\providecommand{\href}[2]{\texttt{#2}}
\providecommand{\urlalt}[2]{\href{#1}{#2}}
\providecommand{\doi}[1]{doi:\urlalt{http://dx.doi.org/#1}{#1}}
\providecommand{\bibinfo}[2]{#2}

\bibitemdeclare{book}{Borovkov}
\bibitem{Borovkov}
\bibinfo{author}{A.~A. \surnamestart Borovkov\surnameend}
  (\bibinfo{year}{1998}): \emph{\bibinfo{title}{Mathematical statistics}}.
\newblock \bibinfo{publisher}{Gordon and Breach Science Publishers},
  \doi{10.2307/3619119}.

\bibitemdeclare{book}{Eaton1983}
\bibitem{Eaton1983}
\bibinfo{author}{M.~L. \surnamestart Eaton\surnameend} (\bibinfo{year}{1983}):
  \emph{\bibinfo{title}{Multivariate statistics. A vector space approach}}.
\newblock \bibinfo{publisher}{Wiley}.
\newblock \urlprefix\url{https://projecteuclid.org/euclid.lnms/1196285102}.

\bibitemdeclare{article}{Edelman1973}
\bibitem{Edelman1973}
\bibinfo{author}{Gerald~M. \surnamestart Edelman\surnameend} \&
  \bibinfo{author}{Joseph~A. \surnamestart Gally\surnameend}
  (\bibinfo{year}{2001}): \emph{\bibinfo{title}{Degeneracy and complexity in
  biological systems}}.
\newblock {\sl \bibinfo{journal}{Proceedings of the National Academy of
  Sciences}} \bibinfo{volume}{98}(\bibinfo{number}{24}), pp.
  \bibinfo{pages}{13763--13768}, \doi{10.1073/pnas.231499798}.
\newblock \urlprefix\url{https://www.pnas.org/content/98/24/13763}.

\bibitemdeclare{book}{EhresmannVanbremeersch2007}
\bibitem{EhresmannVanbremeersch2007}
\bibinfo{author}{Andrée \surnamestart Ehresmann\surnameend} \&
  \bibinfo{author}{Jean-Paul \surnamestart Vanbremeersch\surnameend}
  (\bibinfo{year}{2007}): \emph{\bibinfo{title}{Memory Evolutive Systems;
  Hierarchy, Emergence, Cognition}}, \bibinfo{edition}{first} edition.
\newblock {\sl \bibinfo{series}{Studies in
  multidisciplinarity}}~\bibinfo{volume}{4}, \bibinfo{publisher}{Elsevier},
  \doi{10.1016/S1571-0831(06)04001-9}.

\bibitemdeclare{article}{Ehresmann_2019}
\bibitem{Ehresmann_2019}
\bibinfo{author}{Andr\'{e}e \surnamestart Ehresmann\surnameend} \&
  \bibinfo{author}{Jean{-}Paul \surnamestart Vanbremeersch\surnameend}
  (\bibinfo{year}{2019}): \emph{\bibinfo{title}{{MES}: A Mathematical Model for
  the Revival of Natural Philosophy}}.
\newblock {\sl \bibinfo{journal}{Philosophies}}
  \bibinfo{volume}{4}(\bibinfo{number}{1}), p.~\bibinfo{pages}{9},
  \doi{10.3390/philosophies4010009}.

\bibitemdeclare{book}{Hampel86}
\bibitem{Hampel86}
\bibinfo{author}{F.R. \surnamestart Hampel\surnameend}, \bibinfo{author}{E.M.
  \surnamestart Ronchetti\surnameend}, \bibinfo{author}{P.J. \surnamestart
  Rousseeuw\surnameend} \& \bibinfo{author}{W.A. \surnamestart
  Stahel\surnameend} (\bibinfo{year}{1986}): \emph{\bibinfo{title}{Robust
  Statistics: the Approach based on Influence Functions}}.
\newblock \bibinfo{publisher}{John Wiley and Sons, New York},
  \doi{10.1002/9781118186435}.

\bibitemdeclare{article}{khanduri:IEEE}
\bibitem{khanduri:IEEE}
\bibinfo{author}{Prashant \surnamestart Khanduri\surnameend},
  \bibinfo{author}{Dominique \surnamestart Pastor\surnameend},
  \bibinfo{author}{Vinod \surnamestart Sharma\surnameend} \&
  \bibinfo{author}{Pramod~K. \surnamestart Varshney\surnameend}
  (\bibinfo{year}{2019}): \emph{\bibinfo{title}{{Sequential Random Distortion
  Testing of Non-Stationary Processes}}}.
\newblock {\sl \bibinfo{journal}{IEEE Transactions on Signal Processing}}
  \bibinfo{volume}{67}(\bibinfo{number}{21}), pp. \bibinfo{pages}{5450--5462},
  \doi{10.1109/TSP.2019.2940124}.

\bibitemdeclare{book}{Lehmann2005}
\bibitem{Lehmann2005}
\bibinfo{author}{E.~L. \surnamestart Lehmann\surnameend} \&
  \bibinfo{author}{J.~P. \surnamestart Romano\surnameend}
  (\bibinfo{year}{2005}): \emph{\bibinfo{title}{Testing statistical
  hypotheses}}, \bibinfo{edition}{3rd} edition.
\newblock \bibinfo{publisher}{Springer}, \doi{10.1007/0-387-27605-X}.

\bibitemdeclare{article}{RDT}
\bibitem{RDT}
\bibinfo{author}{D.~\surnamestart Pastor\surnameend} \& \bibinfo{author}{Q.-T.
  \surnamestart Nguyen\surnameend} (\bibinfo{year}{2013}):
  \emph{\bibinfo{title}{{Random Distortion Testing and Optimality of
  Thresholding Tests}}}.
\newblock {\sl \bibinfo{journal}{IEEE Transactions on Signal Processing}}
  \bibinfo{volume}{61}(\bibinfo{number}{16}), pp. \bibinfo{pages}{4161 --
  4171}, \doi{10.1109/TSP.2013.2265680}.

\bibitemdeclare{article}{RDTlm}
\bibitem{RDTlm}
\bibinfo{author}{D.~\surnamestart Pastor\surnameend} \& \bibinfo{author}{F.-X.
  \surnamestart Socheleau\surnameend} (\bibinfo{year}{2018}):
  \emph{\bibinfo{title}{Random distortion testing with linear measurements}}.
\newblock {\sl \bibinfo{journal}{Signal Processing}} \bibinfo{volume}{145}, pp.
  \bibinfo{pages}{116 -- 126}, \doi{10.1016/j.sigpro.2017.11.017}.
\newblock
  \urlprefix\url{http://www.sciencedirect.com/science/article/pii/S0165168417304127}.

\bibitemdeclare{article}{Wald1945}
\bibitem{Wald1945}
\bibinfo{author}{A.~\surnamestart Wald\surnameend} (\bibinfo{year}{1945}):
  \emph{\bibinfo{title}{Sequential Tests of Statistical Hypotheses}}.
\newblock {\sl \bibinfo{journal}{The Annals of Mathematical Statistics}}
  \bibinfo{volume}{16}(\bibinfo{number}{2}), pp. \bibinfo{pages}{117--186},
  \doi{10.1214/aoms/1177731118}.
\newblock \urlprefix\url{http://www.jstor.org/stable/2235829}.

\end{thebibliography}
\end{document}